\let\accentvec\vec
\documentclass{llncs}

\let\vec\accentvec

\usepackage{amsmath,amssymb} 
\usepackage{bbm}
\usepackage{dsfont}
\usepackage{verbatim}
\usepackage{graphicx}
\usepackage{tikz}
\usetikzlibrary{positioning,shadings}

\newcommand{\vol}{\operatorname{vol}}

\DeclareMathOperator{\supp}{supp}

\newcommand{\Ccal}{\mathcal{C}}

\newcommand{\Ocal}{\mathcal{O}}

\newcommand{\Mbf}{\mathbf{M}}
\newcommand{\Sbf}{\mathbf{S}}

\newcommand{\ol}{} 

\usepackage{color}

\title{Mode Poset Probability Polytopes} 
\author{Guido Mont\'ufar\inst{1} \and Johannes Rauh\inst{2}}
\tocauthor{Guido Mont\'ufar and Johannes Rauh}
\institute{Max Planck Institute for Mathematics in the Sciences,\\
  Inselstra\ss e 22, 04103 Leipzig, Germany,\\
  \email{montufar@mis.mpg.de},\\
  \and
  Leibniz Universit\"at Hannover,\\
  Welfengarten 1, 30167 Hannover, Germany,\\
  \email{rauh@math.uni-hannover.de}
}

\begin{document}

\maketitle

\begin{abstract}
A mode of a probability vector is a local maximum with respect to some vicinity structure on the set of elementary events. 
The mode inequalities cut out a polytope from the simplex of probability vectors. 
Related to this is the concept of strong modes. 
A strong mode of a distribution is an elementary event that has more probability mass than all its direct neighbors together. 
The set of probability distributions with a given set of strong modes is again a polytope.
We study the vertices, the facets, and the volume of such polytopes depending on the sets of (strong) modes and the vicinity structures.
\end{abstract}

\section{Introduction}
\label{sec:introduction}

Many probability models used in practice are given in a parametric form. 
Sometimes it is useful to also have an implicit description in terms of properties that characterize the probability distributions that belong to the model. 
Such a description can be used to check whether a given probability distribution lies in the model or, otherwise, to estimate how far it lies from the model. 
For example, if a given model has a parametrization by polynomial functions, then one can show that it has a \emph{semialgebraic description}; that is, an implicit description as the solution set of polynomial equations and polynomial inequalities. 
Finding this description is known as the \emph{implicitization} problem, which in general is very hard to solve completely. 
Even if it is not possible to give a full implicit description, it may be possible to confine the model by simple polynomial equalities and inequalities. 
Here we are interested in simple confinements, in terms of natural classes of linear equalities and inequalities. 

We consider polyhedral sets of discrete probability distributions defined by prescribed sets of modes. 
A mode is a local maximum of a probability vector. 
Locality is with respect to a given a vicinity structure in the set of coordinate indices; that is, $x$ is a (strict) mode of a probability vector $p$ if and only if $p_x>p_y$, for all neighbors $y$ of~$x$. 
The vicinity structure depends on the setting.  
For probability distributions on a set of fixed-length strings, it is natural to call two strings neighbors if and only if they have Hamming distance one. 
For probability distributions on integer intervals, it is natural to call two integers neighbors if and only if they are consecutive.  
In general, a vicinity structure is just a graph with undirected edges.

Modes are important characteristics of probability distributions. 
In particular, the question whether a probability distribution underlying a statistical experiment has one
or more modes is important in applications. 
Also, many statistical models consist of ``nice'' probability distributions that are ``smooth'' in some sense. 
Such probability distributions have only a limited number of modes. 
Another motivation for studying modes was given in~\cite{montufar2012does}, where it was observed that mode patterns are a practical way to differentiate between certain parametric model classes. 

Besides from modes, we are also interested in the related concept of strong modes introduced in~\cite{montufar2012does}. 
A point $x$ is a (strict) strong mode of a probability distribution $p$ if and only if $p_x>\sum_{y\sim x}p_y$, where the sum runs over all neighbors $y$ of~$x$. 
Strong modes offer similar possibilities as modes for studying models of probability distributions. 
While strong modes are more restrictive than modes, they are easier to study. 

One observation is: Suppose that $p=\sum_{i=1}^{k}\lambda_{i}p^{i}$ is a mixture of $k$ probability distributions. 
If $p$ has a strict strong mode $x\in V$, then $x$ must be a mode of one of the distributions~$p^{i}$, because if
$p^{i}(x)\le p^{i}(y_{i})$ for some neighbor $y_{i}$ of $x$ for all~$i$, then
$\sum_{i}\lambda_{i}p^{i}(x)\le\sum_{i}\lambda_{i}p^{i}(y_{i})\le\sum_{y\sim x}\sum_{i}\lambda_{i}p^{i}(y)$. 
For example, a mixture of $k$ uni-modal distributions has at most $k$ strong modes. 
Surprisingly, the same statement is not true for modes: A mixture of $k$ product distributions may have more than $k$ modes~\cite{montufar2012does}. 
Still, the number of modes of a mixture of product distributions is bounded, although this bound is not known in general. 
As another example, in~\cite{montufar2012does} it was shown that a restricted Boltzmann machine with $m$ hidden nodes 
and $n$ visible nodes, where $m<n$ and $m$ is even, does not contain probability distributions with certain patterns of $2^{m}$ strict strong modes. 

In this paper we derive essential properties of (strong) mode polytopes, depending on the vicinity structures and the considered patterns of (strong) modes. 
In particular, we describe the vertices, the facets, and the volume of these polytopes. 
It is worth mentioning that mode probability polytopes are closely related to order and poset polytopes. 
We describe this relation at the end of Section~\ref{sec:order-polytopes}.

\bigskip
This paper is organized as follows: 
In Section~\ref{sec:polytope-modes} we study the polytopes of modes and in Section~\ref{sec:strong-modes} the polytopes of strong modes.

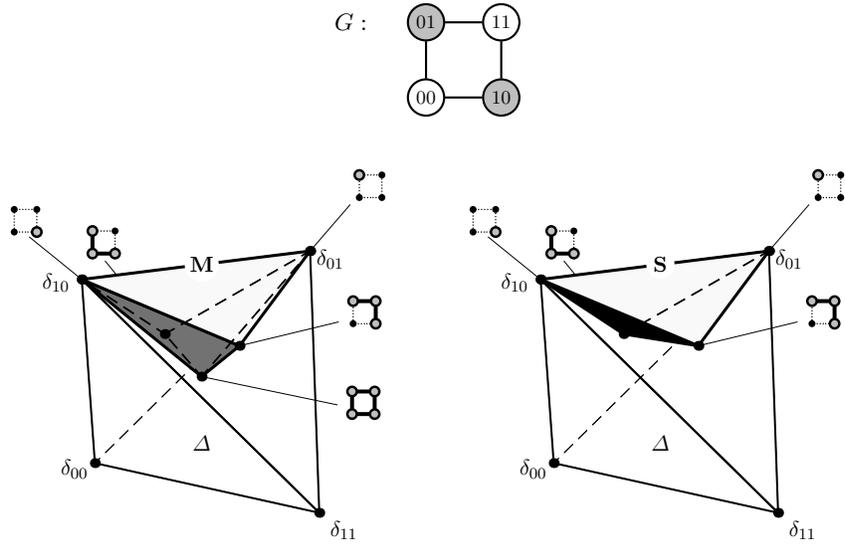
\begin{figure}[t]
	\centering
    \begin{tikzpicture}
      \node at (-1,1) {$G:$};
      \tikzstyle{Cgray}=[draw=black, scale = .8,circle, minimum size=4mm, thick]
      \node (X1) at (0,1) [Cgray, fill = lightgray] {$\!01\!$};
      \node (X2) at (0,0) [Cgray, fill = white] {$\!00\!$};
      \node (X3) at (1,0) [Cgray, fill = lightgray] {$\!10\!$};
      \node (X4) at (1,1) [Cgray, fill = white] {$\!11\!$};
      \draw[thick] (X1) -- (X2) -- (X3) -- (X4) -- (X1);
    \end{tikzpicture}
    \bigskip

	\includegraphics[clip=true,trim=0cm 1cm 0cm 0cm, width=\textwidth]{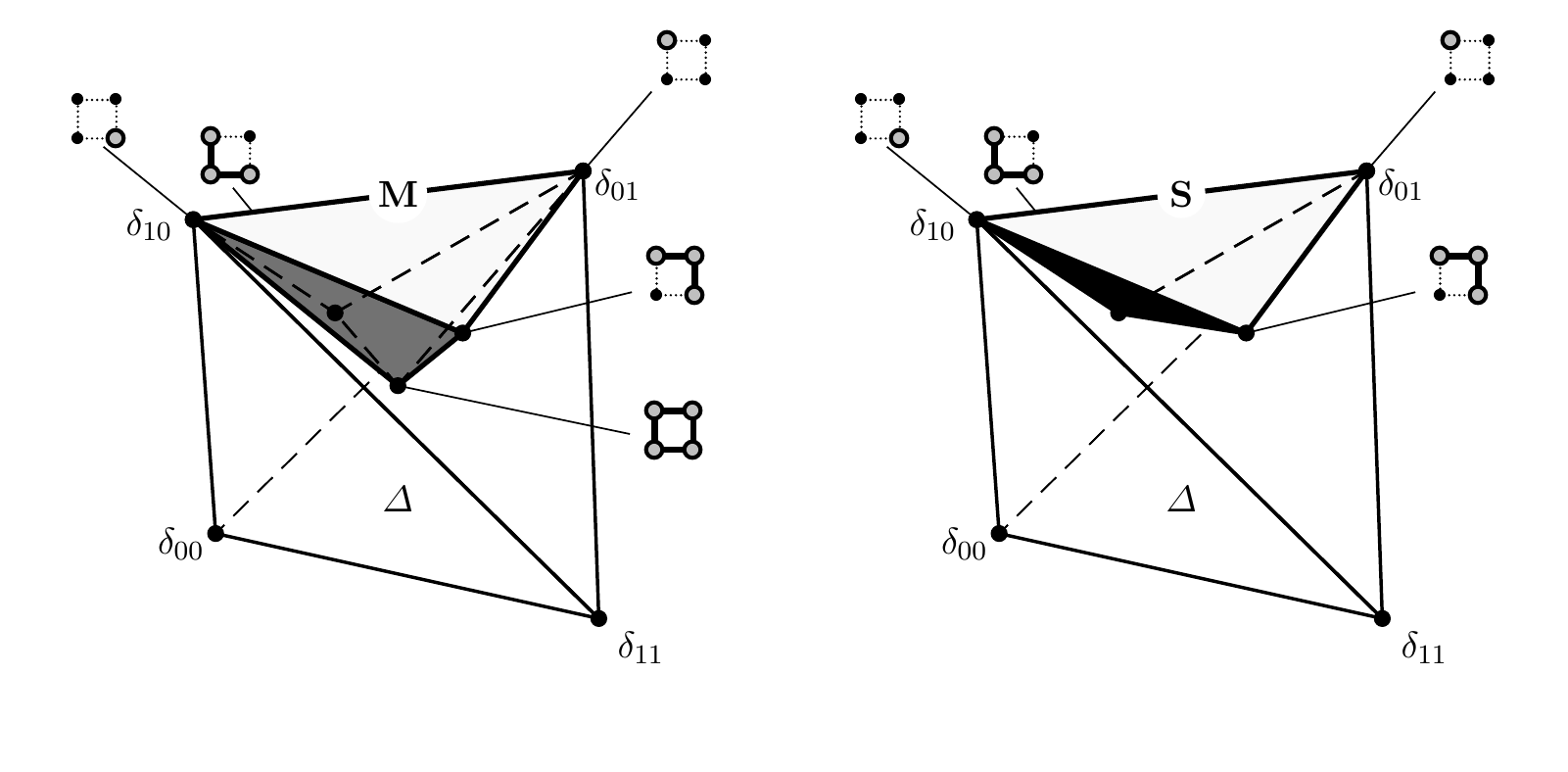}
	\caption{Above: The graph $G$ from Examples~\ref{ex:square-modes} and~\ref{ex:quadrangle-strong-modes}, with $\Ccal$ marked in gray. 
      Below: The corresponding polytopes $\ol{\mathbf{M}}(G,\Ccal)$ and $\ol{\mathbf{S}}(G,\Ccal)$. 
      Each vertex of these polytopes is a uniform distribution supported on a subset of~$G$, as explained in Propositions~\ref{prop:vertices-M} and~\ref{proposition:verticesstrongmodes}.  }
        \label{fig:quadrangle}
\end{figure}

\section{The polytope of modes}
\label{sec:polytope-modes}

We consider a finite set of elementary events $V$ and the set of probability distributions on this set, $\Delta(V)$. 
We endow $V$ with a vicinity structure described by a graph. 
Let $G=(V,E)$ be a simple graph (i.e., no multiple edges and no loops).  For any $x,y\in V$, if $(x,y)\in E$ is an edge in $G$, we write $x\sim y$.  
Since we assume that the graph is simple, $x\sim y$ implies $x\neq y$.

\begin{definition}
A point $x\in V$ is a \emph{mode} of a probability distribution $p\in \Delta(V)$ if $p_x\ge p_y$ for all~$y\sim x$. 
\end{definition}

\begin{definition}
	\label{definition:modepoly}
    Consider a subset $\Ccal\subseteq V$. 
    The \emph{polytope of $\Ccal$-modes} in $G$ is the set $\ol\Mbf(G,\Ccal)$ of all probability distributions
    $p\in\Delta(V)$ for which every $x\in\Ccal$ is a mode. 
\end{definition}
The set $\ol\Mbf(G,\Ccal)$ is always non-empty, since it contains the uniform distribution. 
It is a polytope, because it is a closed convex set defined by finitely many linear inequalities and, as a
subset of $\Delta(V)$, it is bounded. 
We are interested in the properties of this polytope, depending on $G$ and $\Ccal$. 

Recall that a set of vertices of a graph is \emph{independent}, if it does not contain two adjacent elements.  If
$\Ccal$ is not independent, then $\ol\Mbf(G,\Ccal)$ is not full-dimensional as a subset of $\Delta(V)$; that is,
$\dim\ol\Mbf(G,\Ccal)<\dim(\Delta(V))=|V|-1$.  For, if $x,y\in\Ccal$ are neighbors, then the defining equations of
$\ol\Mbf(G,\Ccal)$ imply that $p_x\ge p_y\ge p_x$; that is, any $p\in\ol\Mbf(G,\Ccal)$ satisfies~$p_x=p_y$.  In the
following we will ignore this degenerate case and assume that the set of modes is independent.

In some applications, for example those mentioned in the introduction, it is more natural to study \emph{strict modes};
i.e.~points $x\in V$ with $p_x>p_y$ for all~$y\sim x$.  A description of the set of distributions with prescribed strict
modes is easy to obtain from a description of~$\ol\Mbf(G,\Ccal)$.

\begin{example}
  \label{ex:square-modes}
  Let $G$ be a square with vertices $V=\{00,01,10,11\}$ and edges $E=\{(00,01),(00,10),(01,11),(10,11)\}$.
  The polytope $\ol\Mbf(G,\Ccal)$ for~$\Ccal=\{01,10\}$ is given in Figure~\ref{fig:quadrangle}.
\end{example}

\subsubsection{Vertices.}

We have defined $\ol\Mbf(G,\Ccal)$ by linear inequalities (H-representation). 
Next we determine its vertices (V-representation).
For any non-empty $W\subseteq V\setminus\Ccal$ and $y\in V$ write $y\sim W$ if $y\sim x$ for some~$x\in W$.
Moreover, let $N_{\Ccal}(W) = \{ y\in\Ccal: y\sim W\}$ (this is the set of declared modes which are neighbors of~$W$),
and let $e_{\Ccal}^{W}$ be the uniform distribution on $N_{\Ccal}(W)\cup W$.

\begin{proposition}
  \label{prop:vertices-M}
  \mbox{}
  \begin{enumerate}
  \item
    $\ol\Mbf(G,\Ccal)$ is the convex hull of $\{ e_{\Ccal}^{W} : \emptyset\neq W\subseteq V\setminus\Ccal\} \cup \{
    \delta_{x} : x\in\Ccal\}$, where $\delta_{x}$ denotes the point distribution concentrated on~$x$.
  \item
    For any $x\in\Ccal$, the distribution $\delta_{x}$ is a vertex of $\ol\Mbf(G,\Ccal)$.
  \item
    $e_{\Ccal}^{W}$ is a vertex of $\ol\Mbf(G,\Ccal)$ iff for any $x,y\in W$, $x\neq y$, there is a path
    $x=x_{0}\sim x_{1}\sim\dots\sim x_{r}=y$ in $G$ with $x_{0},x_{2},\dots \in W$ and $x_{1},x_{3},\dots \in 
    N_{\Ccal}(W)$. 
  \end{enumerate}
\end{proposition}
\begin{proof}
  Clearly, for every non-empty $W\subseteq V\setminus\Ccal$, the vector $e_{\Ccal}^{W}$ belongs to
  $\ol\Mbf(G,\Ccal)$, and the same is true for the vectors $\delta_{x}$ with $x\in\Ccal$ ($\Ccal$ is
  independent).  Next we show that each $p\in\ol\Mbf(G,\Ccal)$ can be written as a convex combination of $\{ e_{\Ccal}^{W}
  : \emptyset\neq W\subseteq V\setminus\Ccal\} \cup \{ \delta_{x} : x\in\Ccal\}$.  We do induction on the cardinality of
  $W := \supp(p)\setminus\Ccal$.  If $|W|=0$, then $p\in\Delta(\Ccal)$ is a convex combination of $\{ \delta_{x} :
  x\in\Ccal\}$.  Now assume $|W|>0$.  Let $\lambda = \min\{ p_x : x\in W\}$.
  Then, $p - \lambda e_{\Ccal}^{W}\ge 0$ (component-wise) and $\sum_{x}(p_{x} - \lambda
  e_{\Ccal}^{W}(x))=(1-\lambda)$.  Therefore, 
  \begin{equation*}
    p' := \frac{1}{1-\lambda}(p - \lambda e_{\Ccal}^{W}) \in\Delta(V).
  \end{equation*}
  Moreover, one checks that $p'\in\ol\Mbf(G,\Ccal)$.  By
  definition, $\supp(p') \setminus \Ccal \subsetneq \supp(p)\setminus\Ccal$.  By
  induction, $\supp(p')$ is a convex combination of $\{ e_{\Ccal}^{W} : \emptyset\neq W\subseteq V\setminus\Ccal\} \cup
  \{ \delta_{x} : x\in\Ccal\}$, and so the same is true for~$p$.

  It remains to check which elements of $\{ e_{\Ccal}^{W} : \emptyset\neq W\subseteq V\setminus\Ccal\} \cup \{
  \delta_{x} : x\in\Ccal\}$ are vertices of~$\ol\Mbf(G,\Ccal)$.  Since $\delta_{x}$ is a vertex of $\Delta(V)$, it is also
  a vertex of~$\ol\Mbf(G,\Ccal)$.  Let $W\subset V\setminus\Ccal$ be non-empty.
  Call a path such as in the statement of the proposition an \emph{alternating path}.  Suppose that there is no
  alternating path from $x$ to $y$ for some~$x,y\in W$. Let $W_{1}=\{z\in W:\text{ There is an alternating path from $x$
    to~$z$}\}$ and let $W_{2} = W\setminus W_{1}$.  Then $W_{1}, W_{2}$ are non-empty, and $N_{\Ccal}(W_{1})\cap \tilde
  N_{\Ccal}(W_{2})$ is empty.  Hence $e_{\Ccal}^{W}$ is a convex combination of $e_{\Ccal}^{W_{1}}$ and
  $e_{\Ccal}^{W_{2}}$, and $e_{\Ccal}^{W}$ is not a vertex.

  Let $W$ be a non-empty subset of $V\setminus\Ccal$ such that any pair of elements of $W$ is connected by an alternating path.  
  To show that $e_{\Ccal}^{W}$ is a vertex, for any different non-empty set $W'\subseteq V\setminus\Ccal$ we need to find a face of
  $\ol\Mbf(G,\Ccal)$ that contains $e_{\Ccal}^{W}$ but not $e_{\Ccal}^{W'}$.  If there exists $x\in W'\setminus W$, then  $e_{\Ccal}^{W'}(x) > 0 = e_{\Ccal}^{W}(x)$. 
  Hence, $e_{\Ccal}^{W}$ lies on the face of $\ol\Mbf(G,\Ccal)$ defined by $p_{x}\ge 0$, but $e_{\Ccal}^{W'}$ does not. 
  Otherwise, $W'\subsetneq W$.  Let $x'\in W\setminus W'$ and $y'\in W'\neq\emptyset$.  By assumption, there exists an
  alternating path from $x'$ to~$y'$ in~$W$.  On this path, there exist $x\in W\setminus W'$ and $y\in\Ccal$ with $y\sim
  x$ and $y\in N_{\Ccal}(W')$.  Therefore, $e_{\Ccal}^{W'}(y) - e_{\Ccal}^{W'}(x) > 0 = e_{\Ccal}^{W}(y) - e_{\Ccal}^{W}(x)$. 
\qed
\end{proof}

\begin{corollary}
  \label{cor:dim-modes}
$\ol\Mbf(G,\Ccal)$ is a full-dimensional sub-polytope of~$\Delta(V)$.
\end{corollary}
\begin{proof}
  The convex hull of $\{\delta_{x}:x\in\Ccal\}\cup\{e^{\{y\}}_{\Ccal} : y\in V\setminus\Ccal\}$ is a $(|V|-1)$-simplex
  and a subset of $\ol\Mbf(G,\Ccal)$. 
\qed
\end{proof}

\subsubsection{Facets.}

$\ol\Mbf(G,\Ccal)$ is defined, as a subset of~$\Delta(V)$, by the 
inequalities
\begin{align*}
p_{x}&\ge 0, && \text{for all }x\in V, && \text{(positivity inequalities)} \\
p_{x}&\ge p_{y}, && \text{for all }x\in\Ccal\text{ and }y\sim x. && \text{(mode inequalities)}
\end{align*}
Next we discuss, which of these inequalities define facets.

\begin{proposition}
\mbox{}
  \begin{enumerate}
  \item For any $x\in V\setminus\Ccal$, the positivity inequality $p_{x}\ge 0$ defines a facet.
  \item If $x\in \Ccal$, then $p_{x}\ge 0$ defines a facet iff $x$ is isolated in~$G$.
  \item For any $x\in\Ccal$ and $y\sim x$, the mode inequality $p_{x}\ge p_{y}$ defines a facet.
  \end{enumerate}
\end{proposition}
\begin{proof}
  1. The inequality $p_{x}\ge 0$ defines a facet of the subsimplex from the proof of Corollary~\ref{cor:dim-modes}, and
  hence also of~$\ol\Mbf(G,\Ccal)$.

  2. If $x$ is isolated, then $x$ is a mode of any distribution.  Therefore, $\ol\Mbf(G,\Ccal) =
  \ol\Mbf(\Ccal\setminus\{x\})$, and the statement follows from 1.

  Otherwise, suppose there exists $y\in V$ with $x\sim y$.  Since $\Ccal$ is independent, $y\notin\Ccal$.
  Then $p_{x} = (p_{x} - p_{y}) + p_{y}$; that is, the inequality $p_{x}\ge 0$ is implied by the inequalities
  $p_{x}\ge p_{y}$ and $p_{y}\ge 0$, and $p_{x}\ge 0$ defines a sub-face of the facet $p_{y}\ge 0$, which is a strict
  sub-face, since it does not contain $\delta_{x}$.  Therefore, $p_{x}\ge 0$ does not define a facet itself.

  3. Let $W :=\{ z\in\Ccal : z\sim y\}\setminus\{x\}$.  The uniform distribution on $W\cup\{y\}$ satisfies all defining
  inequalities of $\ol\Mbf(G,\Ccal)$, except~$p_{x}\ge p_{y}$. 
\qed
\end{proof}

\subsubsection{Triangulation and volume.}

The polytope $\ol\Mbf(G,\Ccal)$ has a natural triangulation that comes from a natural triangulation of~$\Delta(V)$. 
Let $N=|V|$ be the cardinality of~$V$.  For any bijection $\sigma:\{1,\dots,N\}\to V$ let
\begin{equation*}
  \Delta_{\sigma} = \{ p\in\Delta(V) : p_{\sigma(i)}\le p_{\sigma(i+1)}\text{ for } i=1,\dots,N-1 \}.
\end{equation*}
Clearly, the $\Delta_{\sigma}$ form a triangulation of~$\Delta(V)$. 
In particular, 
$\Delta(V) = \bigcup_{\sigma}\Delta_{\sigma}$ and $\vol(\Delta_{\sigma}\cup\Delta_{\sigma'}) =
\vol(\Delta_{\sigma})+\vol(\Delta_{\sigma'})$ whenever $\sigma\neq\sigma'$.
\begin{lemma}
  \label{lem:triangulation}
  Let $\Sigma(G,\Ccal)$ be the set of all bijections $\sigma:\{1,\dots,N\}\to V$ that satisfy $\sigma^{-1}(x)<\sigma^{-1}(y)$ for
  all $y\in\Ccal$ and $x\sim y$.  Then $\ol\Mbf(G,\Ccal) = \bigcup_{\sigma\in\Sigma(G,\Ccal)}\Delta_{\sigma}$.
\end{lemma}
\begin{proof}
  If $\sigma\in\Sigma$ and $p\in\Delta_{\sigma}$, then $p\in\ol\Mbf(G,\Ccal)$ by definition.  Conversely, let
  $p\in\ol\Mbf(G,\Ccal)$.  Choose a bijection $\sigma:\{1,\dots,N\}\to V$ that satisfies the following: 
  \begin{enumerate}
  \item $p_{\sigma(i+1)}\ge p_{\sigma(i)}$ for $i=1,\dots,N-1$,
  \item If $x\in\Ccal$ and $y\sim x$, then $\sigma^{-1}(x)\le\sigma^{-1}(y)$.
  \end{enumerate}
Clearly, $\sigma\in\Sigma$, and $p\in\Delta_{\sigma}$. 
 \qed
\end{proof}

\begin{corollary}
  $\vol(\ol\Mbf(G,\Ccal)) = \frac{|\Sigma|}{|V|!}\vol(\Delta(V))$.
\end{corollary}
\begin{proof}
  All simplices $\Delta_{\sigma}$ have the same volume. 
  Moreover, $\vol(\Delta_{\sigma}\cap\Delta_{\sigma'})=0$ for $\sigma\neq\sigma'$.  Thus,
  $\vol(\ol\Mbf(G,\Ccal))=|\Sigma|\vol(\Delta_{\sigma})$ and~$\vol(\Delta(V))=|V|!\vol(\Delta_{\sigma})$.
\qed
\end{proof}

It remains to compute the cardinality of~$\Sigma(G,\Ccal)$.  It is not difficult to enumerate $\Sigma(G,\Ccal)$ by
iterating over the set~$V$. 
However, $\Sigma(G,\Ccal)$ may be a very large, and so, enumerating it can take a very long time. 
In fact, this is a special instance of the problem of counting the number of linear extensions of a partial order (see
below); a problem which in many cases is known to be $\# P$-complete~\cite{CountingLinearExtensions}. 
In our case, a simple lower bound is $|\Sigma(G, \Ccal)|\geq |\Ccal|! |V\setminus \Ccal|!$ (equality holds only when $G$ is a complete bipartite graph and $\Ccal$ is one of the maximal independent sets).

\subsubsection{Relation to order polytopes.}
\label{sec:order-polytopes}

  The results in this section can also be derived from results about order polytopes.
  To explain this, it is convenient to slightly generalize our settings.  Instead of looking at a graph $G$ and an
  independent subset $\Ccal$ of nodes, consider a partial order $\succeq$ on~$V$ and let
  \begin{equation*}
    \ol\Mbf(\succeq) := \{ p\in\Delta(V) : p_x \ge p_y\text{ whenever }x\succeq y \}.
  \end{equation*}
  The polytope $\ol\Mbf(G,\Ccal)$ arises in the special case where $\succeq$ is defined by
  \begin{equation*}
    x \succeq y \quad :\Longleftrightarrow \quad x\sim y\text{ and }x\in\Ccal.
  \end{equation*}
  The relation $\succeq$ defined in this way from~$G$ and~$\Ccal$ is a partial order precisely if~$\Ccal$ is
  independent.
  Our results about vertices, facets and volumes directly generalize to~$\ol\Mbf(\succeq)$. 
  We omit further details at this point. 

  The \emph{order polytope} of a partial order arises by looking at subsets of the unit hypercube instead of subsets of
  the probability simplex (see~\cite{Stanley86:Poset_polytopes} and references):
  \begin{equation*}
    \Ocal(\succeq) := \{ p\in[0,1]^{V} : p_x \ge p_y\text{ whenever }x\succeq y \}. 
  \end{equation*}
  One can show that $\ol\Mbf(\succeq)$ is the vertex figure of~$\Ocal(\succeq)$ at the vertex~$0$.  This observation allows to transfer the results
  from~\cite{Stanley86:Poset_polytopes} to~$\ol\Mbf(G,\Ccal)$.

\section{The polytope of strong modes}
\label{sec:strong-modes}

\begin{definition}
A point $x\in V$ is a \emph{strong mode} of a probability distribution $p\in\Delta(V)$ if $p_x\ge\sum_{y\sim x}p_y$. 
\end{definition}

\begin{definition}
Consider a subset $\Ccal\subseteq V$. 
The \emph{polytope of strong $\Ccal$-modes} in~$G$ is the set $\ol\Sbf(G,\Ccal)$ all probability distributions $p\in\Delta(V)$ 
for which every $x\in\Ccal$ is a strong mode. 
\end{definition}
Again, in applications one may be interested in \emph{strict strong modes} that are characterized by strict inequalities
of the form $p_x>\sum_{y\sim x}p_y$. 

If $x\sim y$ for two strong modes of $p\in\Delta(V)$, 
then $p_x=p_y$ and $p_z=0$ for all other neighbors $z$ of $x$ or $y$. 
In order to avoid such pathological cases, in the following we always assume that $\Ccal$ is an independent subset of $G$. 

\begin{example}
  \label{ex:quadrangle-strong-modes}
  Consider the graph from Example~\ref{ex:square-modes}.
  For~$\Ccal=\{01,10\}$, the polytope $\ol\Sbf(G,\Ccal)$ is given in Figure~\ref{fig:quadrangle}.
\end{example}

Again, we are interested in the vertices of the polytope~$\ol\Sbf(G,\Ccal)$. 
For any $x\in V$ let $N_{\Ccal}(x) = \{ y\in\Ccal: y\sim x \}$ (this is the set of strong modes which are neighbors of $x$) and let $f_{\Ccal}^{x}$ be the uniform distribution on $N_{\Ccal}(x)\cup\{x\}$. 

\begin{proposition}
\label{proposition:verticesstrongmodes}
If $\Ccal$ is independent, then 
$\ol\Sbf(G,\Ccal)$ is a $(|V|-1)$-simplex with vertices $f_{\Ccal}^{x}$, $x\in V$.
\end{proposition}
\begin{proof}
  To see that $\{ f_{\Ccal}^{x} : x\in V\}$ is linearly independent, observe that the matrix with columns
  $f_{\Ccal}^{x}$ is in tridiagonal form when $V$ is ordered such that the vertices in $\Ccal$ come before the vertices
  in $V\setminus\Ccal$. 
  Therefore, the probability distributions $f_{\Ccal}^{x}$ span a $(|V|-1)$-dimensional simplex. 

  It is easy to check that $f_{\Ccal}^{x}\in\ol\Sbf(G,\Ccal)$ for any~$x\in V$.  It remains to prove that any
  $p\in\ol\Sbf(G,\Ccal)$ lies in the convex hull of $\{ f_{\Ccal}^{x} : x\in V\}$.  We do induction on the cardinality
  of $W := \supp(p)\setminus\Ccal$.  If $|W|=0$, then $p\in\Delta(\Ccal)$ is a convex combination of $\{ \delta_{x} :
  x\in\Ccal\}=\{ f_{\Ccal}^{x} : x\in\Ccal\}$.
  Otherwise, let $x\in W$.  Then
  \begin{equation*}
    p' := \frac{1}{1-p_x}(p - p_x f_{\Ccal}^{x})\in\Delta(V),
  \end{equation*}
  since $p\in\ol\Mbf(G,\Ccal)$.  Moreover, $p'\in\ol\Mbf(G,\Ccal)$. 
  The statement now follows by induction, since $\supp(p')\setminus\Ccal = W\setminus\{x\}$.
\qed
\end{proof}

\begin{proposition}
  The facets of $\ol\Sbf(G,\Ccal)$ are  $p_x\ge \sum_{y\sim x} p_y$ for all $x\in\Ccal$ and $p_x\ge 0$ for all $x\in V\setminus\Ccal$. 
\end{proposition}
\begin{proof}
It is easy to verify that each of the faces defined by these inequalities contains $|V|-1$ vertices. 
\qed
\end{proof}

\begin{proposition}
  $\displaystyle \vol(\ol\Sbf(G,\Ccal)) 
	=  \Big( \prod_{x\in V} \frac{1}{|N_\Ccal(x)| +1} \Big) \vol(\Delta(V))$. 
\end{proposition}
\begin{proof}
After rearrangement of columns, the matrix 
	\begin{equation*}
(f^x_\Ccal)_{x\in V} = 
	\left( ( \delta_x)_{x\in\Ccal},  
	\left( \tfrac{1}{|N_\Ccal(x)|+1}\mathds{1}_{N_\Ccal(x)}\right)_{x\in V\setminus \Ccal, x\sim\Ccal} ,  
	(\delta_x)_{x\in V\setminus\Ccal, x\not\sim\Ccal} \right) 
	\end{equation*}
is in upper triangular from, with diagonal elements $\tfrac{1}{|N_\Ccal(x)|+1}$, $x\in V$. 
The statement now follows from the next Lemma~\ref{proposition:volumesimplex}. 
\qed
\end{proof}

\begin{lemma}
	\label{proposition:volumesimplex}
	Let $\Delta=\operatorname{conv} \{e_0,\ldots, e_d\}$ be the standard $d$-simplex in $\mathbb{R}^{d+1}$ and 
	let $s_0,\ldots, s_d\in \Delta$. Then the $d$-volume of $S=\operatorname{conv}\{s_0,\ldots, s_d\}$ satisfies 
	\begin{equation*}
	\vol(S) 
	= |\det(s_0,\ldots, s_d)| \vol(\Delta) . 
	\end{equation*}
\end{lemma}
\begin{proof}
    The $(d+1)$-volume of the parallelepiped spanned by $s_0,\ldots, s_d\in\mathbb{R}^{d+1}$ is $|\det(s_0,\ldots, s_d)|$. 
	The volume of an $n$-simplex with vertices $v_0,\ldots, v_n$ in $\mathbb{R}^n$ is $\frac{1}{n!} |\det(v_1-v_0, \ldots, v_n-v_0)|$. 
	Hence the volume of the $(d+1)$-simplex $P$ with vertices $(0,s_0,\ldots, s_d)$ is $\vol(P) = \frac{1}{(d+1)!}|\det(s_0,\ldots, s_d)|$. 
	Note that $P$ is a pyramid over $S$ of height $h=\frac{1}{\sqrt{d+1}}$. 
    Thus $\vol(P)=\frac{h}{d+1}\vol(S)$.
	The volume of the regular $d$-simplex is $\vol(\Delta) = \frac{\sqrt{d+1}}{d!}$. 
    The statement follows by combining these formulas.
\qed
\end{proof}

\begin{example}
  Generalizing Examples~\ref{ex:square-modes} and~\ref{ex:quadrangle-strong-modes}, let $G$ be the edge graph of an $n$-cube, such that $V=\{0,1\}^n$ and two points are adjacent if their Hamming distance is one. 

  a) If $\Ccal\subseteq V$ has cardinality $|\Ccal|=k$ and minimum distance $3$, 
then $\ol{\mathbf{S}}$ has $2^{n}$ vertices and volume $\vol(\ol{\mathbf{S}}) = 2^{-k n}\vol(\Delta)$, whereas $\ol{\mathbf{M}}$ has $k (2^n-1) + 2^n -k n$ vertices and volume $\vol(\ol{\mathbf{M}}) = \frac{|\Sigma|}{2^n!}\vol(\Delta) \geq k! 2^{-kn} \vol(\Delta)$. 

  b) If $\Ccal$ is the set of all even-parity strings, then $\ol{\mathbf{S}}$ has $2^n$ vertices and volume $\vol(\ol{\mathbf{S}}) = (n+1)^{-2^{n-1}}\vol(\Delta)$, 
whereas $\ol{\mathbf{M}}$ has $2^{2^{n-1}} - 1 + 2^{n-1}$ vertices and volume $\vol(\ol{\mathbf{M}}) = \frac{|\Sigma|}{2^n!}\vol(\Delta) \geq {\binom{2^n}{2^{n-1}}}^{-1}\vol(\Delta)$. 
For $n=2$ and $n=3$ we have $|\Sigma| = 4$ and $|\Sigma| = 720$. 
The next open case is $n=4$. 
\end{example}

\bibliographystyle{abbrv}
\bibliography{referenzen}{}

\end{document}